\documentclass{article}

\usepackage{graphicx}
\usepackage{amsmath}
\usepackage{amssymb}
\usepackage{amsthm}
\usepackage{tikz}
\usepackage{tikz-cd}
\usepackage[all]{xy}

\usetikzlibrary{patterns}
\usetikzlibrary{decorations}
\usetikzlibrary{decorations.pathreplacing}
\usetikzlibrary{topaths}
\usetikzlibrary{shapes.geometric}
\usetikzlibrary{matrix}
\usepackage{tikz-cd}
\tikzcdset{arrow style=tikz,diagrams={>={Stealth[scale=0.9]}}}

\usepackage{comment}
\usepackage[normalem]{ulem}
\usepackage{hyperref}
\usepackage[margin=3cm]{geometry}
\usepackage{amsthm}
\usepackage{amssymb}
\usepackage{graphicx} 
\newtheorem{theorem}{Theorem}

\title{Universal minimal flows of the homeomorphism groups of pseudo-solenoids are non-metrizable}
\author{Jan Boroński\thanks{Faculty of Mathematics and Computer Science, Jagiellonian University in Krak\'ow,
ul. Łojasiewicza 6, 30-348 Kraków, Poland, e-mail: jan.boronski@uj.edu.pl} \thanks{National Supercomputing Centre IT4Innovations, Division of the University of Ostrava, 30. dubna 22, 70103 Ostrava, Czech Republic, e-mail: jan.boronski@osu.cz} \thanks{This article has been supported by EU funds under the project "Increasing the resilience of power grids in the context of decarbonisation, decentralisation and sustainable socioeconomic development", $CZ.02.01.01/00/23\_021/0008759$, through the Operational Programme Johannes Amos Comenius.} \and Aleksandra Kwiatkowska\thanks{Instytut Matematyczny, Uniwersytet Wrocławski, plac Grunwaldzki 2/4, 50-384 Wrocław, Poland, e-mail: aleksandra.kwiatkowska@math.uni.wroc.pl, and
Institut f\"{u}r Mathematische Logik, Universit\"{a}t  M\"{u}nster, Einsteinstraße 62,
48149 M\"{u}nster, Germany}}
\begin{document}

\maketitle
\begin{abstract}
We note that homeomorphism groups of all pseudo-solenoids, including the pseudo-circle, have non-metrizable universal minimal flows. 
\end{abstract}

 For a topological group $G$, a \emph{$G$-flow} (or just a flow, if $G$ is understood) is a continuous action of $G$ on a compact Hausdorff space. A $G$-flow is \emph{minimal} if every orbit is dense. A \emph{universal minimal flow} of $G$ is a minimal $G$-flow that maps continuously and equivariantly onto every other minimal $G$-flow; such a flow is unique up to isomorphism.

A well known open problem, due to Uspenskij~\cite{U}, is to identify the universal minimal flow of the homeomorphism group of the pseudo-arc. It remains unknown whether this flow is metrizable. In fact it is conjectured to be isomorphic to the evaluation action on the pseudo-arc itself.

A {\it pseudo-solenoid} is a circle-like hereditarily indecomposable continuum\footnote{A continuum is said to be {\it circle-like} if it is an inverse limit of circles. A continuum is indecomposable if it is not a union of two proper nondegenerate subcontinua. It is hereditarily indecomposable if every subcontinuum is indecomposable.} with nontrivial 1-st \v Cech cohomology group.
The pseudo-circle is the unique planar pseudo-solenoid. 
Pseudo-solenoids share key properties with the {\it pseudo-arc}, the unique circle-like hereditarily indecomposable continuum with trivial 1-st \v Cech cohomology group (see \cite{F}). Notably,  any proper nondegenerate subcontinuum of each of the pseudo-solenoids and the pseudo-arc is always the pseudo-arc.

Universal minimal flows have been computed for the homeomorphism groups of several classical continua, including the Lelek fan~\cite{BK}, generalized Ważewski dendrites~\cite{Kwi}, and the universal Knaster continuum~\cite{Iyer-Knaster}. While the universal minimal flow is metrizable in the case of the Lelek fan, it is non-metrizable for the universal Knaster continuum as well as for certain generalized Ważewski dendrites.
Furthermore, the non-metrizability of the universal minimal flow has been established for the homeomorphism group of the Menger universal curve along with many other Peano continua~\cite{BCV}, generalizing the techniques developed for the Hilbert cube and closed manifolds of dimension two or higher~\cite{GTZ}. 

Iyer \cite{SI} showed that the group of homeomorphisms of the universal pseudo-solenoid has non-metrizable universal minimal flow. We note that a much more general result holds true.
\begin{theorem}
 Homeomorphism groups of all pseudo-solenoids have non-metrizable universal minimal flows.
\end{theorem}
\begin{proof}
By \cite[Theorem 9.2]{BCO} all pseudo-solenoids admit a minimal homeomorphism \footnote{A homeomorphism is minimal if all points have dense orbits.} and by \cite[Theorem 4]{Bu} all pseudo-solenoids are almost-chainable. By \cite[Theorem 5]{KR} the homeomorphism group of an almost-chainable continuum with Kelley property, that admits a fixed-point-free homeomorphism, does not have a dense $G_\delta$-orbit. This applies to pseudo-solenoids as all hereditarily indecomposable continua have the property of Kelley \cite[8.5 Theorem]{Kelley}.  Therefore the homeomorphism group of every  pseudo-solenoid admits a minimal flow with all orbits meager. By \cite[Theorem 1.2]{BMT} and \cite[Proposition 14.1]{AKL} our proof is complete.
\end{proof}

For readers' convenience, let us briefly recall why the pseudo-solenoids admit minimal homeomorphisms (see \cite{BCO} for details). The first proof that the pseudo-circle admits such a homeomorphism is due to Handel in \cite{Ha}, where he constructed
a $C^\infty$-diffeomorphism of the annulus $\mathbb{A}$, with an attracting minimal pseudo-circle $\mathcal{C}$. This diffeomorphism $h$ can be made area-preserving, if one does not require the pseudo-circle to be attracting. In both cases $h|\mathcal C$ is a minimal homeomorphism of the pseudo-circle. The $h|\mathcal C$ is a uniform limit of a sequence of rational rotations (periodic homeomorphisms) $R_n:A_n\to A_n$, where each $A_n$ is an annular neighborhood of $\mathcal C$, with $A_n\subset A_{n+1}$ and $\bigcap_{n=1}^\infty A_n=\mathcal C$. It has a well-defined irrational rotation number, although it is not semi-conjugate to an irrational rotation. 
To obtain a minimal homeomorphism of the $P$-adic pseudo-solenoid, with $P=(p_i:i\in\mathbb{N})$ a sequence of primes, one first notes that the pseudo-circle admits an $n$-fold self-covering for any $n\in\mathbb{N}$. This follows from considering the $n$-fold covering map $\tau_n:\mathbb{A}\to\mathbb{A}$. As observed in \cite[Example 1]{Heath}, $\mathcal C_n=\tau_n^{-1}(\mathcal C)$ is also a pseudo-circle, so $\tau_n|\mathcal C_n$ gives rise to an $n$-fold self-covering of $\mathcal C$. 
In addition, $h$ lifts to a homeomorphism $h_n$ satisfying $\tau_n\circ h_n|\mathcal{C}_n=h\circ \tau_n|\mathcal{C}_n$, with $h_n|\mathcal{C}_n$ minimal. The minimality of $h_n|\mathcal{C}_n$  follows from Handel's construction involving rational rotations: each $h_n$ can be approximated by periodic rotations of crookedly nested annuli obtained from lifting the nested annuli in the base $\mathbb{A}$, and therefore $h_n|\mathcal{C}_n$ is again a homeomorphism as in Handel's construction, in particular, it is minimal \cite[Proposition 2.2]{BCO}. Now let $\mathcal{C}=\mathcal{C}_0$ and consider the inverse limit $\mathcal{C}_P=\lim_{\leftarrow}\{\mathcal C_{p_n},\tau_{p_n}:n\in\mathbb{N}\}$. Since each $C_{p_n}$ is hereditarily indecomposable and circle-like, so is $\mathcal{C}_P$. Since the 1-st \v Cech cohomology group of $C_P$ is the direct limit $\lim_\to\{\mathbb{Z},\tau_{p_n}:n\in\mathbb{N}\}$, it follows that $\mathcal{C}_P$ is the $P$-adic pseudo-solenoid. Finally, let $h_0=h$ and note that the homeomorphism $h_P:\mathcal{C}_P\to\mathcal{C}_P$ given by $h_P((x_n:n\in\mathbb{N}))=(h_{p_n}(x_n):n\in\mathbb{N})$ is minimal, since each $h_n$ is. 
\begin{equation}
\xymatrix{
 \mathcal{C}_0 \ar[d]^{h_0}
& \mathcal{C}_1 \ar[l]^{\tau_1}\ar[d]^{h_1}
& \cdots\ar[l]^{\tau_2}
& \mathcal{C}_{i-1}\ar[l]\ar[d]^{h_{i-1}}
& \mathcal{C}_i \ar[l]^{\tau_i}\ar[d]^{h_i}
& \cdots\ar[l]
&\mathcal{C}_P \ar[d]^{h_P} \\
  \mathcal{C}_0
& \mathcal{C}_1 \ar[l]^{\tau_1}
&\cdots\ar[l]^{\tau_2}
& \mathcal{C}_{i-1}\ar[l]
& \mathcal{C}_i \ar[l]^{\tau_i}
& \cdots\ar[l]
& \mathcal{C}_P}
\end{equation}
It is also important to note that pseudo-solenoids, unlike the pseudo-arc, are not homogeneous. In fact,  every $\mathcal{C}_P$ has uncountably many orbits under the action of its homeomorphism group. This was proved by Kennedy and Rogers \cite[Theorem 9]{KR} for the pseudo-circle, but their arguments extend to all the pseudo-solenoids. Indeed, suppose by contradiction that $\mathcal{C}_P$ has only countably many orbits under the action of its homeomorphism group. Then one of these orbits is second category in $\mathcal{C}_P$, and consequently second category in itself.  By Effros's Theorem \cite{Ef} $\mathcal{C}_P$ has a $G_\delta$-orbit. But every orbit is dense, since $\mathcal{C}_P$ admits a minimal homeomorphism, leading to a contradiction.


\begin{thebibliography}{99}
\bibitem{AKL} {\sc O. Angel, A. S. Kechris and R. Lyons}, {\em Random orderings and unique ergodicity of automorphism
groups}, {\bf J. Eur. Math. Soc. (JEMS) 16} (2014), no. 10, 2059--2095.

\bibitem{BK} {\sc D. Barto\v sov\'a{} and A. Kwiatkowska}, {\em The universal minimal flow of the homeomorphism group of the Lelek fan}, {\bf Trans. Amer. Math. Soc.  371} (2019), no.~10, 6995--7027.

\bibitem{BCV} {\sc G. Basso, A. Codenotti and A. Vaccaro}, {\em Surfaces and other Peano continua with no generic chains}, {\bf Duke Math. J.  174} (2025), no.~14, 3135--3196.

\bibitem{BMT} {\sc I. Ben Yaacov, J. Melleray and T. Tsankov}, {\em Metrizable universal minimal flows of Polish groups have a comeagre orbit.} {\bf Geom. Funct. Anal. 27} (2017), no. 1, 67--77. 

 \bibitem{BCO} {\sc J. P. Boroński, A. Clark and  P. Oprocha}, {\em New exotic minimal sets from pseudo-suspensions of Cantor systems.} {\bf J. Dynam. Differential Equations 35} (2023), 1175--1201.

\bibitem{Bu} {\sc C. E. Burgess}, {\em Homogeneous continua which are almost chainable.} {\bf Canadian J. Math. 13} (1961), 519--528.

\bibitem{Ef} {\sc E. G. Effros}, {\em Transformation groups and $C^*$-algebras}, {\bf Ann. of Math. 81} (1965),
38--55.
\bibitem{F} {\sc L. Fearnley}, {\em Classification of all hereditarily indecomposable circularly chainable continua}, {\bf Trans. Amer. Math. Soc.  168} (1972), 387--401. 

\bibitem{GTZ} {\sc Y. Gutman, T. Tsankov and A. Zucker}, {\em Universal minimal flows of homeomorphism groups of high-dimensional manifolds are not metrizable}, {\bf Math. Ann.  379} (2021), no.~3-4, 1605--1622.

\bibitem{Ha} {\sc M. Handel}, {\em A pathological area preserving $C^\infty$ diffeomorphism of the plane.} \textbf{Proc.
Amer. Math. Soc. 86} (1982), 163--168.

\bibitem{Heath} {\sc J. W. Heath}, {\em Weakly confluent, 2-to-1 maps on hereditarily indecomposable continua}, {\bf Proc. Amer. Math. Soc. 117} (1993), 569--573.

\bibitem{SI}  {\sc S. Iyer}, {\em Universal minimal flows of homeomorphism groups of continua}, arXiv:2606.20407

\bibitem{Iyer-Knaster} {\sc S. Iyer}, {\em The homeomorphism group of the universal Knaster continuum}, {\bf Israel J. Math.  271} (2026), no.~2, 501--536.

\bibitem{Kelley} {\sc J. L. Kelley}, {\em Hyperspaces of a continuum.} {\bf Trans. Amer. Math. Soc. 52 } (1942), 22--36.

\bibitem{KR} {\sc J. Kennedy and  J. T. Rogers Jr.}, {\em Orbits of the pseudocircle.} {\bf Trans. Amer. Math. Soc. 296} (1986), 327--340.


\bibitem{Kwi} {\sc A. Kwiatkowska}, {\em Universal minimal flows of generalized Wa\.zewski dendrites}, {\bf J. Symb. Log.~83} (2018), no.~4, 1618--1632.

\bibitem{U} {\sc V.~V. Uspenskij}, {\em On universal minimal compact $G$-spaces}, {\bf Topology Proc.  25} (2000), Spring, 301--308.

\end{thebibliography}
\end{document}